\documentclass[11pt,a4paper]{article}
\usepackage{mathrsfs}
\usepackage{amsmath,amsthm,color,graphicx}
\usepackage{amssymb} \usepackage{verbatim}
\usepackage{pict2e}
\usepackage[noblocks]{authblk}
\usepackage{hyperref}

\newtheorem{myDef}{Definition}
\newtheorem{theorem}{Theorem}

\newtheorem{lemma}{Lemma}
\newtheorem{construction}{Construction}

\newtheorem{claim}{Claim}
\theoremstyle{definition}
\newtheorem*{remark}{Remark}

\newcommand{\floor}[1]{\left\lfloor{#1}\right\rfloor}

\def\F{\mathcal F}
\def\f{\mathbb F}

\providecommand{\keywords}[1]
{
 \small \begin{center}
 \textit{Keywords:} #1
 \end{center}
}

\newcommand{\abs}[1]{\left\lvert{#1}\right\rvert}

\begin{document}

\title{Intersecting families of polynomials over finite fields}

\author[1]{Nika Salia}
\author[2,3]{Dávid Tóth}

\affil[1]{King Fahd University of Petroleum and Minerals, Dhahran, Saudi Arabia}
\affil[2]{HUN-REN Alfr\'ed R\'enyi Institute of Mathematics, Budapest,
Hungary}
\affil[3]{Department of Computer Science and Information Theory, Budapest University of Technology and Economics, Budapest, Hungary}

\date{}
\maketitle

\begin{abstract}

This paper establishes an analog of the Erdős–Ko–Rado theorem to polynomial rings over finite fields, affirmatively answering a conjecture of C. Tompkins.

A $k$-uniform family of subsets of a set of finite size $n$ is $\ell$-intersecting if any two subsets in the family intersect in at least $\ell$ elements. The study of such intersecting families is a core subject of extremal set theory, tracing its roots to the seminal 1961 Erdős–Ko–Rado theorem, which establishes a sharp upper bound on the size of these families. Finite fields, or Galois fields, represent a well-studied object in discrete mathematics since the 19th century. This paper explores an extension of the Erdős–Ko–Rado theorem to polynomial rings over finite fields.

Specifically, we determine the largest possible size of a family of monic polynomials, each of degree $n$, over a finite field $\f_q$, where every pair of polynomials in the family shares a common factor of degree at least  $\ell$. We establish that the upper bound for this size is $q^{n-\ell}$ and characterize all extremal families that achieve this maximum size.

Further extending our study to triple-intersecting families, where every triplet of polynomials shares a common factor of degree at least $\ell$, we prove that only trivial families achieve the corresponding upper bound. Moreover, by relaxing the conditions to include polynomials of degree at most $n$, we affirm that only trivial families achieve the corresponding upper bound.
\end{abstract}

{\keywords{polynomial rings over finite fields; combinatorial number theory; Erdős-Ko-Rado theorem}}

\section{Introduction}

In 1961, Erdős, Ko  and Rado~\cite{erdos1961intersection} 
established a foundational result in extremal set theory by proving that for any finite family $\F$ of $k$-subsets of a set with $n$ elements, $[n] := \{1, 2, \dots, n\}$, where every pair of subsets shares at least $\ell\leq k$ elements, the size of $\F$ is bounded above by $\binom{n-\ell}{k-\ell}$, provided $n$ is sufficiently large. This theorem has profoundly influenced the field due to its wide-ranging applications. 
Note that this problem is an instance of Erdős matching 
conjecture~\cite{erdos1965problem}.

The upper bound of the Erdős-Ko-Rado theorem is sharp and is achieved by a family 
with a
unique structure.
For large enough $n$ all extremal families are isomorphic to 
\[
\left\{\{n-\ell+1,\dots,n\}\cup A: A\in \binom{[n-\ell]}{k-\ell}\right\},
\]
where $\binom{B}{k}$ denotes all $k$-subsets of a set $B$. 

Hilton and Milner~\cite{hilton1967some} obtained stability of the Erdős-Ko-Rado theorem by establishing the maximal size of a $1$-intersecting family, 
where no single element intersects all members of the family, and described
the structure of such extremal families. 
Subsequently, using the Delta-system method, Frankl, Kostochka, and Mubayi~\cite{kostochka2017structure} strengthened the Hilton-Milner theorem.

Intersecting families are 
very well-studied.
Meagher and Purdy~\cite{meagher2011erdHos} extended the Erdős-Ko-Rado theorem for unbounded multisets, i.e., where each element has a uniform infinite bound,  using the graph homomorphism method, see also~\cite{anderson2002combinatorics,anderson1988erdos}.
Furthermore, Füredi, Gerbner, and Vizer~\cite{furedi2015discrete} expanded this result to an $\ell$-intersecting version for unbounded multisets, where the structure of the extremal set was determined by Meagher and Purdy~\cite{meagher2016intersection}.
Very recently Liao, Lv, Cao, and Lu~\cite{liao2024erdHos} extended the Erdős-Ko-Rado theorem for multisets with uniform multiplicity. 
Thus this result may be considered as an extension of both the Erdős-Ko-Rado theorem and the result of  Meagher and Purdy~\cite{meagher2011erdHos}.
These authors further refined their work by extending the Hilton-Milner theorem to uniformly bounded multisets~\cite{liao2023nontrivial}.

Inspired by the Erdős-Ko-Rado theorem, researchers have explored similar problems in different settings. 
This includes studies on Erdős-Ko-Rado problems within permutation groups~\cite{ellis2011intersecting,li2020erd,keller2024t}. 
Meagher, Spiga, and Tiep \cite{meagher2016erdHos} further extended these ideas by adapting the theorem to intersecting sets of permutations in finite $2$-transitive groups. For a deeper understanding of versions of the Erdős-Ko-Rado theorem with algebraic nature see the book by Godsil and Meagher~\cite{godsil2016erdos}.

Aguglia, Csajb{\'o}k and Weiner~\cite{aguglia2022intersecting} improved Adriaensen's~\cite{adriaensen2022erdHos} result and obtained a version of the  Erdős-Ko-Rado theorem for graphs of functions over a finite field. 
Gadouleau, Mariot and 
Mazzone~\cite{gadouleau2024maximal} proved an analog of the  
Erdős-Ko-Rado theorem for binary polynomials with pairwise linear common factors.

Building on some of the results discussed above, Tompkins~\cite{CT} conjectured a version of the  Erdős-Ko-Rado theorem for polynomials over finite fields. 
While the challenges associated with polynomials over the complex field directly mirror the results of Meagher and Purdy for multisets, studying polynomials over finite fields introduces unique characteristics. 
The varied structures of extremal families primarily drive this interest in these settings.
Specifically, we examine monic polynomials of fixed degrees over finite fields.

\subsection{The extremal constructions}

Let $\f_q[x] $ be the polynomial ring over the finite field $ \f_q $ with $ q $ elements. A family $ \F \subseteq \f_q[x] $ is termed \emph{$ \ell $-intersecting} if for every pair of polynomials $ p_1 $ and $ p_2 $ in $ \F $ we have $ \deg(\gcd(p_1, p_2)) \geq \ell $.
For a family $\F$ of polynomials and a polynomial $p$, we denote by $\F\cdot p$ the family generated from $\F$ by multiplying each element by $p$. Similarly, we 
write $\frac{\F}{p}:=\F\cdot \frac{1}{p}$ 
for any non-zero $p$ if all elements of
$\F$ are divisible by $p$.

A family that possesses a given property $\mathcal{P}$ and has the largest possible size is called an \emph{extremal family} for $\mathcal{P}$. If the property $\mathcal{P}$ is understood from the context, we simply refer to the family as extremal without explicitly mentioning $\mathcal{P}$.
In the following, we restrict ourselves to
\emph{monic} polynomials
and examine extremal $\ell$-intersecting families of them.
The corresponding results for 
general classes are easy consequences of the subsequent 
theorems. 

Motivated by the beauty of the topic,
Tompkins~\cite{CT} conjectured that 
a family consisting of all degree $n$ multiples
of a fixed degree $\ell$ polynomial
is  an  extremal $\ell$-intersecting 
 family of 
degree $n$-polynomials,
and asked if these examples 
are the only extremal 
ones.
In the following, a family described above will be referred to as 
\emph{trivial}.
\begin{construction}\label{const:Trivial}
    Given any non-negative 
    integers $n$ and $\ell$ with $n\geq \ell$ and a prime power $q$, consider a monic polynomial $g\in \f_q[x]$ of degree $\ell$, and set
    \[ 
         \F_{Triv}(g):=
         \left\{g\cdot p: p~\text{is a  monic polynomial over }\f_q,
         ~\deg(p)=n-\ell\right\}. 
    \]
\end{construction}

Note that $ \F_{Triv}(g)$ is an $\ell$-intersecting 
family of monic polynomials of degree $n$ and 
the size of it is $q^{n-\ell}$, and hence, by
Theorem~\ref{thm:extremal_number} below, 
it is of maximum size.
The extremal families with the same size and different structure will 
be referred to as \emph{non-trivial}.

As we will soon see, there are infinitely many non-trivial
extremal constructions for any $q$. To describe all of them, 
we first note that once $\F$ is 
an extremal $\ell$-intersecting family of degree $n$ polynomials and $g\in \f_q[x]$
is any monic polynomial, then $\F\cdot g$ is an  
extremal
$(\ell+\deg(g))$-intersecting family with elements of degree $n+\deg(g)$
(this follows again from Theorem~\ref{thm:extremal_number}).
Similarly, if a monic 
polynomial $g$ divides every element of the extremal
family $\F$,
then  $\frac{\F}{g}$ is still extremal. Loosely speaking,
it is enough to describe the extremal families up to polynomial 
multiples.

\begin{myDef}
    A maximal size $\ell$-intersecting family $\F$ of degree $n$ polynomials over 
    $\f_q$ is \emph{primary} if no non-unit polynomial divides every
    element of $\F$. 
\end{myDef}

To make our previous statement precise, it is clear  that the 
description of all extremal families is equivalent to the description 
of all primary examples. 
Note that a trivial extremal family is primary if and only if 
it is $0$-intersecting and consists of all
monic polynomials of a fixed degree.

We give an example of primary extremal families. 
Let $q>2$ be a prime power, $\ell=q-2$, $n=q-1$, 
$H_1:=\prod_{a\in \f_q}(x-a)$  and
\[
\F_{1}:=
\left\{\frac{H_1}{(x-a)}:a\in \f_q\right\}.
\]    
It is easy to see that 
$\F_{1}$ 
is a non-trivial, $\ell$-intersecting set of degree $n$ polynomials  of size $q^{n-\ell}=q$.
Note that $H_1$ is a monic polynomial that equals the least common multiple of all polynomials of degree one. 
This example can easily be generalized to develop a broad class of constructions.
In the following construction, the constant $d$ can be interpreted as the difference $n-\ell$.

\begin{construction}\label{Const:Primary_COnstruction}
 Let $d$ be a positive integer and $H_d$ be the monic polynomial equal to the least common multiple of all 
 degree $d$
 polynomials
 over $\f_q$.
 We set
\begin{equation*}
\F_{d}:=
\left\{ \frac{H_d}{p}:p~\text{is a  monic polynomial over }\f_q,
~\deg(p)=d
\right\}.
\end{equation*}
    
\end{construction}

Note that  $\F_{d}$ is a $(\deg(H_d)-2d)$-intersecting  
family of degree $\deg(H_d)-d$ monic polynomials. 
The size of $\F_{d}$ is $q^d=q^{(\deg(H_d)-d)-(\deg(H_d)-2d)}$,
hence - as we 
will soon prove - it is of maximum size.
Moreover, apart from the case $q=2$, 
$d=1$, $\F_d$
is non-trivial, and since no non-unit polynomial divides 
every element of the family,
 it is also primary. 
Remark that this construction can be viewed as similar to the one in the original Erdős–Ko–Rado theorem when the ground set is smaller relative to the sizes of the sets and their intersection size.

One of our main goals is to show that the aforementioned primary 
constructions are the only possible ones with one single exception 
that we give as the closure of this introductory section:
\begin{construction}\label{Const:exceptional_construction}
If $q=2$, $\ell=1$ and $n=3$, then
 \begin{equation*}
\F=\{x^2(x+1), ~x(x+1)^2,~ x(x^2+x+1),~ (x+1)(x^2+x+1)\}    
\end{equation*}
is a maximum size primary $1$-intersecting family
of degree $3$ monic polynomials in $\f_2[x]$ 
consisting of $q^2=q^{n-\ell}$
elements. Observe that $\F$ differs from the family $\F_2$ defined
in Construction~\ref{Const:Primary_COnstruction}.
\end{construction}

\section{Main Results}
In this section, we give our main results about extremal
$\ell$-intersecting polynomial families. 
First, we determine the size of
such sets, then we continue by examining
the possible structures of them.

\begin{theorem} \label{thm:extremal_number}
Given any non-negative integers $ n $ and $ \ell $ with $n\geq \ell$ and a prime power $ q $, consider $ \F \subseteq \f_q[x] $, an $\ell$-intersecting family of monic polynomials of degree $n$.
Then we have
\[
\left|\F\right| \leq q^{n-\ell}.
\]
If $ n > 2\ell$, then any extremal $\F$ is trivial with the sole exception 
given in Construction~\ref{Const:exceptional_construction}.
\end{theorem}

It follows now that the examples given in the
introductory section (including 
Construction~\ref{const:Trivial},~\ref{Const:Primary_COnstruction},~\ref{Const:exceptional_construction}) 
are all extremal. 
Also, by Theorem~\ref{thm:extremal_number}, we have that all extremal families for which $n>2\ell+1$ are trivial, while for $n=2\ell+1$ there is a unique non-trivial extremal family. 
A natural question is
if the constant $2$ can be 
lowered with an additional constraint.
In fact, by the following theorem, for every constant $c$ greater than $1$ and large enough $\ell$ all extremal families are trivial.
Specifically, only a finite number of non-trivial extremal examples
exist for $n > c\ell$.
It is easy to see that the constant $c$ must be greater than $1$ by 
Construction~\ref{Const:Primary_COnstruction}. 
Indeed, for every 
$d=n-\ell$ and every $q$, we have a primary non-trivial extremal family~\eqref{Const:Primary_COnstruction} (except for the case $d=1$ and $q=2$).

\begin{theorem} \label{thm:trivial_thm}
For any constant
$c>1$
there is a threshold $\ell_c$ such that any extremal $\ell$-intersecting family $\F \subseteq \f_q[x]$ of monic polynomials of degree $n$ is trivial, provided $\ell > \ell_c$ and  $n > c\ell$. 
The threshold $\ell_c$ can be chosen independently of $q$.
Specifically, only a finite number of non-trivial extremal instances exist with $n > c\ell$, counting over all finite fields $\f_q$. 
\end{theorem}

Next, we describe all the possible extremal $\ell$-intersecting 
families. As it was detailed
in the introductory section,
this is equivalent to the description of 
the primary extremal examples. 
The trivial families were defined in Construction~\ref{const:Trivial}, and all of them 
belong to a $0$-intersecting primary family consisting
of all monic polynomials of a fixed (but otherwise arbitrary) degree.
In the forthcoming theorem, we provide a comprehensive characterization of all primary non-trivial extremal families.

\begin{theorem}\label{thm:primary_examples}
Every primary non-trivial extremal family different from Construction~\ref{Const:exceptional_construction} is a family from Construction~\ref{Const:Primary_COnstruction}.
\end{theorem}

\begin{remark}
As a corollary, we note that once we require any $k\geq 3$ elements of a family $\F$ of degree $n$ monic polynomials to have a common divisor of degree at least $\ell$, then all extremal families are trivial.
Indeed, as such a family is $\ell$-intersecting, and the trivial construction has this property, the size of an extremal family is still $q^{n-\ell}$, and the non-trivial constructions in Theorem~\ref{thm:primary_examples} do not have this stronger property. \end{remark}

Finally, we shortly address another version of the problem where one allows polynomials with
different 
degrees in $\F$.
\begin{theorem}\label{thm:version2}
For any non-negative integer $ \ell $ let $D_\ell$ be a finite set of integers greater than
$\ell$ with $|D_\ell|\geq2$, and let $q$ be a prime power. Consider  
a family $ \F $ of monic polynomials  over $ \f_q $ whose degrees are in $D_\ell$ such that any two
elements
of $ \F $ have a common divisor of degree at least $\ell$. Then
\[
\left|\F\right| \leq \sum_{d\in D_\ell}q^{d-\ell},
\]
where the upper bound is sharp 
and every extremal  $\F$ is trivial.
\end{theorem}

Below we give two different proofs of this theorem. One of them
uses only Theorem~\ref{thm:extremal_number} and
a lemma stated and proved in the course of its proof,
while the other one is an application of 
Theorem~\ref{thm:primary_examples}.

\section{Proofs of Theorems}

The subsequent proofs of the theorems stated above 
have in big part a combinatorial flavor, but
we necessarily use some basic algebraic properties of
the ring $\f_q[x]$ and 
some simple number theoretic results.
Firstly, we will use the fact that $\f_q[x]$ is Euclidean and
hence a unique factorization domain, in particular, every
irreducible element of it is prime. 

Besides, we will need
some lower estimates on
the number $N_q(n)$ of degree $n$ monic irreducibles over $\f_q$. 
This is given by the prime number theorem for polynomials over
finite fields, in fact, we use the  explicit  bound
\begin{equation}\label{Nq_estimate}
    N_q(n)\geq \frac{q^n}{n}-\frac{q^{\frac{n}{2}}}{n}-q^{\frac{n}{3}}
\end{equation}
(see Theorem 2.2 in \cite{rosen2013number} and the estimate above it).
This result is an easy consequence of the well-known formula
\begin{equation}\label{Nq_eq}
    N_q(n)=\frac{1}{n}\sum_{d\mid n}\mu(n/d)q^d
\end{equation}
(see the Corollary to Proposition 2.1 in \cite{rosen2013number}).
From these one deduces easily the following:
\begin{claim}\label{claim:existence_of_irred}
 The number $N_q(n)$ is at least
    $2$ for every positive integer $n$ and prime power $q$, 
    except for the case $q=n=2$.
\end{claim}
This simple but in some situations useful little statement follows 
directly from \eqref{Nq_estimate} in the cases $q\geq4$;
$q=3$, $n\geq3$ and $q=2$, $n\geq 5$, in all 
other cases one may  apply \eqref{Nq_eq}.
We are now ready to prove Theorem~\ref{thm:extremal_number}.

\begin{proof}[\textbf{Proof of Theorem~\ref{thm:extremal_number}}]
  Assume for the sake of contradiction that $ |\F| > q^{n-\ell}$. By the pigeonhole principle, there exist two distinct monic polynomials $ p_1 $ and $ p_2 $ in $ \F $ such that
  the coefficients of $x^{n-1},\dots, x^{\ell}$ are the same for both of the polynomials. 
  Hence we have
\begin{equation*}\label{eq:Deg_p_1-p_2}
    \deg(p_1 - p_2) \leq \ell-1.
\end{equation*}
Since $ \F $ is defined as $ \ell $-intersecting, the degree of the greatest common divisor of polynomials $ p_1 $ and $ p_2 $ in $ \F $ is at least $ \ell $. On the other hand  $ p_1 - p_2$ is a multiple of $ \gcd(p_1, p_2) $, thus $\ell-1\geq \deg(p_1-p_2)\geq \gcd(p_1, p_2)\geq \ell$, a contradiction.

The following statement has a key role in most arguments that follow,
hence we separate it as a lemma.

\begin{lemma}\label{lemma:there_is_pol_divisible_byiorreducible_f}
Let $\F$ be an extremal family of $\ell$-intersecting monic polynomials of degree $n$ over $\f_q$.
Then, for any irreducible monic polynomial $f$ of degree $n-\ell$ within $ \f_q[x] $, there exists a polynomial $ p $ in $ \F $ divisible by $ f $.   
 \end{lemma}
\begin{proof}
As $\f_q[x]$ is Euclidean, any polynomial $ p $ in $ \F $  can 
be expressed as
\[
p = f \cdot s + r,
\]
where $ s $ is a monic polynomial of degree $ \ell $ and $ r $ is a polynomial with 
degree at most $ n-\ell-1 $. 
Suppose, for the sake of contradiction, that $ r \neq 0 $ for all $ p\in\F $.
Given that the cardinality of $ \F $ is $ q^{n-\ell} $ and the number of possible non-zero residues modulo $ f $ is one less, the pigeonhole principle implies that there must be two distinct polynomials, $ p_1 $ and $ p_2 $, in $ \F $ with the same residue modulo $ f $. Hence, $ f $ divides $ p_1 - p_2 $.
As neither $ p_1 $ nor $ p_2 $ is divisible by $ f $, it follows that $ \gcd(p_1, p_2) $ divides $ \frac{p_1-p_2}{f} $, a non-zero polynomial of degree  at most
$ n-1-(n-\ell)=\ell-1$, 
contradicting the assertion that $ \F $ is $ \ell $-intersecting.
\end{proof}

Suppose $ n > 2\ell$ and let $ \F $ be an $\ell$-intersecting family of polynomials with the maximum size $ q^{n-\ell} $. According to Claim~\ref{claim:existence_of_irred}, there exist at least two distinct irreducible monic polynomials of degree $ n-\ell $
except for the case $q=2$, $\ell=1$, and $n=3$ that will be handled at the end of the proof.

Consider $ f_1 $, an irreducible monic polynomial of degree $ n-\ell $ in $ \f_q[x] $. As per Lemma~\ref{lemma:there_is_pol_divisible_byiorreducible_f}, there exists a polynomial $ p_1 $ in $ \F $ that is divisible by $ f_1 $. Let $ p_1 = f_1 \cdot g_1 $, where $ g_1 $ is a monic polynomial of degree $ \ell $. Given that $ \F $ is $ \ell $-intersecting, any polynomial in $ \F $  must be divisible by either $ f_1 $ or $ g_1 $.
Therefore, either all polynomials in $\F$ are divisible by $g_1$ making $\F$ trivial, or there exists another polynomial $p_2 \in \F$ such that $p_2 = f_1 \cdot g_2$ with $g_2 \neq g_1$.

By Claim~\ref{claim:existence_of_irred}, 
there is another irreducible monic polynomial
$ f_2 $ of degree $ n-\ell $ in $ \f_q[x] $ distinct from $f_1$. 
According to Lemma~\ref{lemma:there_is_pol_divisible_byiorreducible_f}, there exists a polynomial $ p $ in $ \F $ that is divisible by $ f_2 $. Given that $ p $ is not divisible by $ f_1 $ due to $ n-\ell > \ell $, it follows that $ p $ must be divisible by $ g_1 $ and $ g_2 $. However, since $ g_1 $ and $ g_2 $ are distinct monic polynomials each of degree $ \ell $, their co-division of $ p $ is impossible, leading to a contradiction.

Assume now that $q=2$, $\ell=1$ and $n=3$. The above argument still gives that in a non-trivial maximal size example there are two polynomials divisible by the unique irreducible degree $2$ factor 
$x^2+x+1$.
It is straightforward to conclude that $\F$ must be
the family given in Construction~\ref{Const:exceptional_construction}, i.e.
\[
\F=\{x^2(x+1),~ x(x+1)^2,~ x(x^2+x+1), ~(x+1)~(x^2+x+1)\}.
\]

\end{proof}

\begin{proof}[\textbf{Proof of Theorem~\ref{thm:trivial_thm}}]
We prove by induction, starting with the assertion that the statement holds when
 $n>c_1 \ell$ for
 $c_1=2$ by Theorem~\ref{thm:extremal_number}. 
 We then show that this implies that the statement holds for $c_{i+1}\ell<n\leq c_i\ell$ and hence 
 for  $c_{i+1}\ell<n$ where
 $c_{i+1}=2-\frac1{c_i}$ for any integer $i\geq 1$.
 Since $c_i$ is strictly decreasing and converges to $1$, 
 this implies the statement of the theorem for every $c>1$.

The subsequent argument uses the existence of $k$ distinct irreducible monic polynomials of degree 
$n-\ell$, 
for some $k\in\mathbb{N^+}$.
We have that  by 
\eqref{Nq_estimate} for any fixed $k$
except when
 $q$ and $n-\ell$ are small. 
 If $d:=n-\ell$ is fixed, then
\[
c_{i+1}\ell<n\leq c_i\ell\Longleftrightarrow
\frac{d}{c_i-1}\leq\ell<\frac{d}{c_{i+1}-1},
\]
i.e. for any $k\in\mathbb{N}^+$  there are $k$ distinct irreducibles except for finitely many
 triples $(q,\ell,n)$.

Initially, we assume that for every pair of distinct 
irreducible polynomials $f,\tilde{f}\in\f_q[x]$ 
of degree $n-\ell$, there exists a polynomial $p\in \F$ such that $f\mid p$ and $\tilde{f}\nmid p$.
Let $f_1,\dots, f_k$ be distinct irreducible polynomials over $\f_q[x]$ of degree $n-\ell$, and
let us apply the assumption  with the choices $f=f_1$, $\tilde{f}=f_2$,
and $f=f_2$, $\tilde{f}=f_1$.
Then,
there exist $p_1,p_2\in\F$ such that $p_1=f_1\cdot g_1$, 
$p_2=f_2\cdot g_2$, $f_2\nmid p_1$
and $f_1\nmid p_2$. Since the degree of  $\gcd(p_1,p_2)$ is at least $\ell$, we have $\gcd(p_1,p_2)=\gcd(g_1,g_2)=g_1=g_2$. 
Now if $p_3=f_1\cdot g_3\in\F$ and $f_2\nmid p_3$, 
then $g_3=\gcd(p_2,p_3)=g_2=g_1$, i.e. $p_3=p_1$.
In other words, $p_1$ is the only polynomial in $\F$ that is a multiple of $f_1$ but not of $f_2$.

If $p_1$ is the only 
multiple of $f_1$, then every polynomial of $\F$ is a multiple $g_1$, and $\F$ is trivial.
Otherwise, we have a polynomial in $\F$ that is a multiple of $f_1 \cdot f_2$.
Let us fix such a $p_0=f_1\cdot f_2\cdot g_0\in\F$ and consider an element
 $p\in\F$ for which $f_1\nmid p$. 
 Then $\gcd(p,p_1)=g_1=g_2$, i.e. $p=g_2\cdot h$. 
 Also, $\gcd (p,p_0)=f_2\cdot g_0$ and $\gcd(f_2, g_2)=1$, thus $h=f_2$, that is $p=p_2$.
 Thus the only polynomial that is not divisible by $f_1$ is $p_2$.
 Repeating the argument for $f_1$ and an irreducible polynomial from $\{f_3,\dots, f_k\}$
 we conclude that $f_2\cdot f_3\cdots f_k$ divides $p_2$, which is impossible if
 \[
 k>\frac{c_{i}}{c_{i+1}-1}+1> \frac{n}{n-\ell}+1
 \]
 since then $(k-1)(n-\ell)>n$ holds.

 Next assume that there are  irreducibles $f_1$, $f_2$ of degree $n-\ell$ such that 
 $f_1\mid p\Rightarrow f_2\mid p$ holds
for every $p\in\F$. 
That is, if $p\in\F$ and $f_1\mid p$, then $p=f_1\cdot f_2\cdot g$ for some
monic polynomial $g\in \f_q[x]$ of degree $n-2(n-\ell)=2\ell -n$.
Let us fix such a $p\in\F$ (that exists by Lemma~\ref{lemma:there_is_pol_divisible_byiorreducible_f}).
If there is a $\hat{p}\in\F$ such that $f_2\nmid \hat{p}$
and hence $f_1\nmid \hat{p}$, then the greatest common divisor of $p$ and $\hat{p}$ must
divide $g$ and therefore its degree is at most $2\ell-n<\ell$ that is impossible.
Consequently, any element of $\F$ is divisible by $f_2$.
Note that if $\deg(f_2)=n-\ell=\ell$ (that is possible in the first step of the induction),
then $\F$ is trivial, so in the following we may assume $n-\ell<\ell$.
Let us consider now the family
\[
\F'=\frac{\F}{f_2}.
\]
Each polynomial of $\F'$ has degree $\ell$ and as $\F$ is $\ell$-intersecting, 
$\F'$ is $(2\ell-n)$-intersecting. 
Since  $1-2c_{i}+c_{i}c_{i+1}=0$,
 we have $\ell= c_{i}(2\ell-c_{i+1}\ell)>c_{i}(2\ell-n)$ and if
 $i>1$, then $2\ell-n$ is arbitrary large  once $\ell$ is large enough since  
 $2\ell-n\geq(2-c_i)\ell$ holds.
 Hence $\F'$ and therefore $\F$ must be trivial 
 for a large enough $\ell$ by the inductive assumption. 

 In the situation of the previous paragraph, we must handle the case 
 $i=1$ separately, since then $c_1=2$ and hence $2-c_1=0$ hold.
 But by Theorem~\ref{thm:extremal_number},
 the inequality $\ell>2(2\ell-n)$ implies that $\F'$
 is trivial except for the case $q=2$, $2\ell-n=1$ and $\ell=3$,
 that implies $n=5$. Hence if $\ell$ is bigger than $3$,
 then $\F'$ is trivial.

Finally, assume that $\F$ is non-trivial. Since $c_{i+1}\ell <n\leq c_i\ell $, there are only finitely 
many possibilities for $n$ and $\ell$. Now for any fixed $n$ and $\ell$, 
if $q$ and hence $k$ is big enough, then the argument above yields a non-trivial 
$(2\ell-n)$-intersecting family $\F'$ of degree $\ell$, and since
$\ell> c_i(2\ell-n)$, we have finitely many possibilities for $\F'$, hence
we have finitely many choices for $q$ as well. Lastly, for a fixed triple $(q,\ell,n)$
there are finitely many maximal size sets of $\ell$-intersecting polynomials of degree $n$,
hence the induction step and the proof is complete.
\end{proof}

\begin{proof}[\textbf{Proof of Theorem~\ref{thm:primary_examples}}]
Let us fix a prime power $q$ and a difference $d>0$, and assume that $\F$ is a maximal size primary 
(non-trivial) $\ell$-intersecting family of degree $n$ polynomials in $\f_q[x]$ such that $n-\ell=d$.
First, we assume the existence of at least two distinct irreducible polynomials of degree $d$,
that is, we exclude the case $q=2=d$ that will be handled later.

Notice that we already know a lot about the structure of $\F$ 
based on the proof of Theorem~\ref{thm:trivial_thm}. Indeed, by 
the argument of the fifth paragraph of that proof, 
if $f_1$ and
$f_2$ are distinct irreducibles of degree $n-\ell$ and $f_1\mid p\Rightarrow f_2\mid p$ for
any $p\in\F$, then any element of $\F$ is divisible by $f_2$ contradicting the assumption that
$\F$ is primary.

Hence for each pair of distinct irreducible monic polynomials $f_i$ and $f_j$ of degree $n-\ell$, $\F$ 
must contain a polynomial which is divisible by $f_i$ but not 
by $f_j$. 
In this case, as it was deduced in the fourth paragraph of the proof 
of  Theorem~\ref{thm:trivial_thm}, we get
that the structure of $\F$ is the following. 
Let $f_1,f_2,\dots, f_k$ be all irreducible monic polynomials of degree $d=n-\ell$  ($k\geq 2$) and let $F$ be the product of them. 
 We must have 
 \begin{equation}\label{formula:F_structure}
 \F=\left \{\frac{F}{f_i}\cdot G :1\leq i\leq k \right \} \cup 
 \left\{F\cdot \frac{G}{p}:  p\mid G,~ p~\text{monic},~ 
 \deg(p)=d=n-\ell   \right\},
 \end{equation}
for some monic polynomial $G$ coprime to $F$.

 Note that in the second set above we need to take all the appropriate divisors of $G$ because
 $\F$ is of maximal size. 
 Since the first set on the right-hand side of (\ref{formula:F_structure})
 has cardinality $N_q(d)$, the second one must contain $|\F|-N_q(d)=q^d-N_q(d)$
 polynomials, i.e. $G$ must be divisible by \emph{all} polynomials of degree $d$ that are not  irreducible.
 Thus, we have
\[
\F = \left\{ \frac{F\cdot G}{p} : p \mid F\cdot G,~ p~\text{monic},~\deg(p) = d = n - \ell \right\}.
\]
Since $F \cdot G$ is divisible by all degree $d$ polynomials, it is a multiple of their least common multiple.
Moreover, since $\F$ is a primary family, $F\cdot G$ in fact equals the monic least common multiple of all degree $d$ polynomials.
Thus $F\cdot G=H_d$ and $\F=\F_d$ as in Construction~\ref{Const:Primary_COnstruction}.

It remains to handle
the case $q=2=d$. 
First note 
that $|\F|=4$. 
By 
Lemma~\ref{lemma:there_is_pol_divisible_byiorreducible_f}, at least 
one element of $\F$ is divisible by $f(x)=x^2+x+1$, but  $\F$ is non-trivial, so
$f$ must divide in fact at least two elements. Also, since $\F$ is primary, $f$ cannot
divide all of the elements.

Assume that $f$ divides exactly $2$ elements of $\F$. Since $\F$ is primary,
no irreducible of degree at least $3$ divides any of its elements,
hence exactly two elements contain only linear factors.
It follows that $f^2$ cannot divide any element as $d=2$. Therefore
\[
\F=\{x^{\alpha_1}(x+1)^{\beta_1},\,x^{\alpha_2}(x+1)^{\beta_2},\,
x^{\alpha_3}(x+1)^{\beta_3}f(x),\,
x^{\alpha_4}(x+1)^{\beta_4}f(x)\}.
\]
As $f\nmid x^{\alpha_i}(x+1)^{\beta_i}$ for $i=1,2$ and $d=2$, we have that
$\alpha_1$ and $\alpha_2$ cannot be smaller than $\alpha_3$ or $\alpha_4$.
Since $\F$ is primary, $\alpha_3$ or $\alpha_4$ must be $0$. Without loss
of generality, we may assume $\alpha_4=0$. By a similar argument,
we get $\beta_3=0$. 
Then the greatest common divisor of the last two polynomials
is $f$, and since $d=2$, we infer $0<\alpha_3=\beta_4\leq 2$.

If $\alpha_3=\beta_4=2$, then
 $\alpha_1+\beta_1=\alpha_2+\beta_2=4$. 
By symmetry we may assume $\alpha_1>\alpha_2$ that implies $\beta_1<\beta_2$.
Since $\ell=4-d=2$, we have $\alpha_1-\alpha_2+\beta_2-\beta_1=2$, i.e. 
$\alpha_1-\alpha_2=1=\beta_2-\beta_1$. Hence $\F$ has the following form:
\[
\F=\{x^{\alpha+1}(x+1)^{\beta},\,x^{\alpha}(x+1)^{\beta+1},\,
x^{2}f(x),\,
(x+1)^{2}f(x)\}.
\]
But both $\alpha$ and $\beta$ must be at least $2$, that is impossible. 

The case $\alpha_3=\beta_4=1$ gives the family from Construction~\ref{Const:exceptional_construction} in a similar way.

Finally, if exactly $3$ elements of $\F$ are 
divisible by $f$, 
then as above, we get
\[
\F=\{x^{\alpha_1}(x+1)^{\beta_1},\,x^{\alpha_2}(x+1)^{\beta_2}f(x),\,
x^{\alpha_3}(x+1)^{\beta_3}f(x),\,
x^{\alpha_4}(x+1)^{\beta_4}f(x)\}.
\]
Again, we can assume $\alpha_4=\beta_3=0$ and $0<\alpha_3=\beta_4\leq 2$.
Since the degrees are the same and the last three polynomials are different,
the only possibility is $\alpha_3=\beta_4=2$, which forces the second polynomial
to be $x(x+1)f(x)$ and the first one to be $x^2(x+1)^2$:
\[
\F=\{x^{2}(x+1)^{2},\,x(x+1)f(x),\,
x^{2}f(x),\,
(x+1)^{2}f(x)\},
\]
that is exactly Construction~\ref{Const:Primary_COnstruction}.
\end{proof}

\begin{proof}[\textbf{Proof of Theorem~\ref{thm:version2}}]
Let $\F$ be an $\ell$-intersecting set of monic polynomials over $\f_q$ with degrees in $D_\ell$, and 
let $\F^d$ denote the set of degree $d$ elements of $\F$.
Then $\F^d$ is still an $\ell$-intersecting set for any $d$, hence $|\F^d|\leq q^{d-\ell}$ 
by Theorem~\ref{thm:extremal_number} and  $|\F|\leq \sum_{d\in D_\ell}q^{d-\ell}$ follows.
If $|\F|$ consists of all multiples of a fixed degree $\ell$ polynomial with degrees in $D_\ell$, 
then $|\F| =\sum_{d\in D_\ell}q^{d-\ell}$, hence the upper bound is sharp.

Assume that for some $\ell$ and $D_\ell$ there exists a non-trivial family $\F$ with
$|\F| =\sum_{d\in D_\ell}q^{d-\ell}$. 
We can choose such $\F$ so that $\ell$ is minimal.
Note that $|\F^d|=q^{d-\ell}$ must
hold for any $d\in D_\ell$, i.e. $\F^d$ is an extremal $\ell$-intersecting family.

Let us set $n=\max (D_\ell)$ and let $f\in\f_q[x]$ be an irreducible polynomial of degree $n-\ell$.
By Lemma~\ref{lemma:there_is_pol_divisible_byiorreducible_f} we have 
$f\cdot g\in \F^n$ for some $g\in\f_q[x]$ of degree $\ell$.
Assume that $p\in\F$, $\deg(p)<n$ and $f\mid p$, i.e. $p=f\cdot h$ where $\deg(h)<\ell$. 
If $\hat{p}\in \F$, then since
$\deg(\gcd(\hat{p},p))\geq \ell$ must hold, we obtain $f\mid \hat{p}$.
That is, every element of $\F$ is a multiple of $f$.
As $\F$ is non-trivial, we must have $n-\ell<\ell$, contradicting  the minimality of $\ell$
since we could consider $\frac{\F}{f}$ instead.
Therefore, $f\nmid p$ must hold and hence $\gcd(f\cdot g,p)=g$ giving $g\mid p$,
and this applies for any $p\in\F$
of degree less than $n$. 

Since $\F$ is non-trivial, there must be an element in $\F$ that is not divisible by $g$,
and by the previous paragraph, this element must be in $\F^n$. But $f\cdot g\in \F^n$,
hence such an element must be divisible by $f$, so there is an element $f\cdot \hat{g}\in\F^n$
with $\deg(\hat{g})=\ell$, $\hat{g}\neq g$. Repeating the argument of the previous paragraph with $\hat{g}$ instead
of $g$, we get that any $p\in\F$
of degree less than $n$ is divisible by $\hat{g}$. Since $\F^d$ is $\ell$-intersecting
of maximal size  for any $d\in D_\ell$ and $\deg(g)=\deg(\hat{g})=\ell$,
we infer $g=\hat{g}$ contradicting the choice of $\hat{g}$.
\end{proof}

\begin{proof}[\textbf{Second proof of Theorem~\ref{thm:version2}}]

Let $\F$ be an extremal $\ell$-intersecting set.
As in the previous proof we have $\abs{\F}=\sum_{d\in D_\ell}q^{d-\ell}$, and $\abs{\F^d}=q^{d-\ell}$ where $\F^d$ denotes the set of degree $d$ polynomials of $\F$ as above.

If all $\F^d$'s are trivial then let the monic polynomials $g_1$ and $g_2$ of degree $\ell$ be the greatest common divisors of $\F^{d_1}$ and $\F^{d_2}$ respectively for some distinct
$d_1,d_2\in D_{\ell}$.
If there are two distinct elements $a_1$ and $a_2$ of $\f_q$, such that the multiplicity of the 
factor $x-a_i$
is at least as large in $g_i$ 
as it is
in $g_{3-i}$ for $i\in\{1,2\}$ then we have
\[
\deg(\gcd(g_1(x-a_1)^{d_1-\ell},g_2(x-a_2)^{d_2-\ell}))=\deg(\gcd(g_1,g_2))\leq \ell.
\]
Since $g_1(x-a_1)^{d_1-\ell},g_2(x-a_2)^{d_2-\ell}$ 
are polynomials of the $\ell$-intersecting family $\F$,
we have  the inequality $\deg(\gcd(g_1,g_2))\geq \ell$ as well
and hence $g_1=g_2$.
Otherwise, without loss of generality, we may assume 
 that for every $a\in \f_q$
the multiplicity of the 
 factor $x-a$
in $g_1$ is strictly greater than in $g_2$. 
Therefore, the polynomial $g_2$ is divisible by an irreducible polynomial $f$ of degree at least two, where the multiplicity of $f$ in $g_2$ exceeds its multiplicity in $g_1$.
This simply implies
$\deg(\gcd(g_1,g_2)) \leq \ell-\deg(f)$ and 
\begin{equation*}\label{eq_g_1=g_2}
 \deg(\gcd(g_1(x-a)^{d_1-\ell},
 g_2f^{\floor{\frac{d_{2}-\ell}{\deg(f)}}}(x-a)^{d_{2}-\ell-\deg(f)
 \floor{\frac{d_{2}-\ell}{\deg(f)}}}))\leq (\deg(f)-1)+ \deg(\gcd(g_1,g_2))<\ell
\end{equation*}
for any  $a\in\f_q$,
that is impossible since $\mathcal{F}$ is $\ell$-intersecting.

Thus 
either we are done or, there is an  $n\in D_{\ell}$ such that $\F^{n}$ is a non-trivial  $\ell$-intersecting family of degree $n$ monic polynomials. 
Let $g$ be the greatest common divisor of all polynomials in $\F^n$ 
with $\deg(g)=\ell_{0}<\ell$. 
Let  us set $n_{1}:=n-\ell_{0}$, 
$\ell_{1}:=\ell-\ell_{0}$ and $\F^{n_{1},\ell_{1}}:=\frac{\F^{n}}{g}$. 
Then, by Theorem~\ref{thm:primary_examples}, either 
\[
\F^{n_1,\ell_1}=\F_{n_1-\ell_1}=
\left\{ \frac{H_{n_1-\ell_1}}{p}:p~\text{is a  monic polynomial over }\f_q,
~\deg(p)=n_1-\ell_1
\right\},
\]
where $H_{n_1-\ell_1}$ is the monic polynomial equal to the least common multiple of all degree ${n_1-\ell_1}$ polynomials over $\f_q$, 
or
\[
\F^{n_1,\ell_1}=\{x^2(x+1), ~x(x+1)^2,~ x(x^2+x+1),~ (x+1)(x^2+x+1)\}.
\]

If there is an $f\in \F$ such that $\deg(f)<n$ then, by the maximality of $\F$, the polynomials $f\cdot x^{n-\deg(f)}$ and $f\cdot (x-1)^{n-\deg(f)}$ are in $\F^n$,
and hence they are multiples of $g$. 
Therefore (as $x^{n-\deg(f)}$ and $(x-1)^{n-\deg(f)}$ are coprime)
$f$ is a multiple of  $g$.

If $\mathcal{F}^{n_1,\ell_1}$ is the family from 
Construction~\ref{Const:exceptional_construction},
and hence $q=2$, $n_1=3$, $\ell_1=1$, then
$\deg (f)=n-1$ since $n-2<\deg(f)<n$ as $\mathcal{F}$ is an $(n-2)$-intersecting family.
That is, $\mathcal{F}^{\deg(f)}$ must contain two elements of
the form $g \cdot p$ with $\deg(p)=2$.
As $p$ has a common factor with every element of $\mathcal{F}^{n_1,\ell_1}$, $p$ must be $x(x+1)$, thus there is no choice for two such polynomials, a contradiction.

If
$\F^{n_1,\ell_1}=\F_{n_1-\ell_1}$, then
assume first that there is a monic irreducible polynomial $p$ of degree $n_1-\ell_1$ which is coprime to
$\frac{f}{g}$.
Note that if $n_1-\ell_1=1$, then $q>2$, since otherwise there is no $\ell_1$-intersecting primary construction of degree $n_1$ polynomials. 
Then
without loss of generality, we may assume $p\neq x$ and $p\neq x-1$. Thus $\frac{f}{g}\cdot x^{n-\deg(f)}$ and $\frac{f}{g}\cdot (x-1)^{n-\deg(f)}$ are coprime to $p$ and by the structure of $\F^n$
they are both equal to 
$\frac{H_{n_1-\ell_1}}{p}$, 
a contradiction.

It follows that
$\frac{f}{g}$ is a multiple of all irreducible polynomials of degree $n_1-\ell_1$.
Let $f_{n_1,\ell_1}$ be an irreducible of degree 
$n_1-\ell_1=n-\ell$, note that it is coprime 
to $\frac{H_{n_1-\ell_1}}{f_{n_1,\ell_1}}$, and we have
\[
\deg\left(\gcd\left(\frac{H_{n_1-\ell_1}}{f_{n_1,\ell_1}}\cdot g, \frac{f}{g}\cdot g \right)\right)<\ell,
\]
since $\frac{f}{g}$ is a multiple of $f_{n_1,\ell_1}$ and  $\deg(f)<n$. 
A contradiction to $\mathcal{F}$ being
$\ell$-intersecting since $f$ and $\frac{H_{n_1-\ell_1}}{f_{n_1,\ell_1}}\cdot g\in \F$.

Finally, if $f\in \F$ such that $\deg(f)>n$, then $\F^{\deg(f)}$ is a trivial family by the previous argument, i.e. it consists
of all degree $\deg(f)$ multiples of a degree $\ell$ polynomial $h$.
Since $\F^{n}$ is non-trivial, hence there 
is a $p\in\F^n$ such
that the degree of
$\hat{h}=\gcd(h,p)$ is
less than $\ell$.
Since $\gcd(p,hx^{\deg(f)-\deg(h)})$ has degree at least $\ell$, 
it follows that $p/\hat{h}$ is a multiple of $x^{\ell-\deg(\hat{h})}$.
 But this is impossible
 since 
 \[
 \deg\left(
    \gcd\left(
    \left(\frac{px^{\deg(f)-n}}
 {\hat{h}x^{\ell-\deg(\hat{h})}}+1\right)h,p
        \right)
 \right)<\ell.
\]
\end{proof}

\section*{Acknowledgments}
We are grateful to Casey Tompkins for encouraging us to work on this project and for his motivating questions and suggestions. His guidance has been instrumental in advancing this work.

We are grateful to the reviewers for their valuable suggestions, which helped to improve the manuscript's clarity. 

The research of Salia was partially supported by the National Research, Development
and Innovation Office NKFIH, grant K132696.  

The research of D. Tóth was supported 
 by the Ministry of
	Innovation and  Technology and the National Research, 
  Development and
		Innovation Office  within the Artificial Intelligence National
		Laboratory of Hungary.
Project no. TKP2021-NVA-02 has been implemented with the support
provided by the Ministry of Culture and Innovation of Hungary from the
National Research, Development, and Innovation Fund, financed under the
TKP2021-NVA funding scheme.
The research was also supported by the MTA–RI Lendület "Momentum" Analytic Number Theory and Representation Theory Research Group and by the NKFIH (National Research, Development and Innovation Office) grant FK 135218.

\bibliographystyle{abbrv}
\bibliography{references.bib}

\begin{thebibliography}{10}

\bibitem{adriaensen2022erdHos}
S.~Adriaensen.
\newblock Erd{\H{o}}s-{K}o-{R}ado theorems for ovoidal circle geometries and
  polynomials over finite fields.
\newblock {\em Linear Algebra and its Applications}, 643:1--38, 2022.

\bibitem{anderson1988erdos}
I.~Anderson.
\newblock An {E}rd{\H{o}}s-{K}o-{R}ado theorem for multisets.
\newblock {\em Discrete mathematics}, 69(1):1--9, 1988.

\bibitem{anderson2002combinatorics}
I.~Anderson.
\newblock {\em Combinatorics of finite sets}.
\newblock Courier Corporation, 2002.

\bibitem{aguglia2022intersecting}
B.~Csajb{\'o}k, A.~Aguglia, and Z.~Weiner.
\newblock Intersecting families of graphs of functions over a finite field.
\newblock {\em Ars Mathematica Contemporanea}, 24(1), 2023.

\bibitem{ellis2011intersecting}
D.~Ellis, E.~Friedgut, and H.~Pilpel.
\newblock Intersecting families of permutations.
\newblock {\em Journal of the American Mathematical Society}, 24(3):649--682,
  2011.

\bibitem{erdos1961intersection}
P.~{Erd\H{o}s}.
\newblock Intersection theorems for systems of finite sets.
\newblock {\em Quart. J. Math. Oxford Ser.(2)}, 12:313--320, 1961.

\bibitem{erdos1965problem}
P.~Erd\H{o}s.
\newblock A problem on independent r-tuples.
\newblock {\em Ann. Univ. Sci. Budapest. E{\"o}tv{\"o}s Sect. Math},
  8(93-95):2, 1965.

\bibitem{furedi2015discrete}
Z.~F{\"u}redi, D.~Gerbner, and M.~Vizer.
\newblock A discrete isodiametric result: the {E}rd{\H{o}}s--{K}o--{R}ado
  theorem for multisets.
\newblock {\em European Journal of Combinatorics}, 48:224--233, 2015.

\bibitem{gadouleau2024maximal}
M.~Gadouleau, L.~Mariot, and F.~Mazzone.
\newblock On maximal families of binary polynomials with pairwise linear common
  factors.
\newblock {\em arXiv preprint arXiv:2405.08741}, 2024.

\bibitem{godsil2016erdos}
C.~Godsil and K.~Meagher.
\newblock {\em Erd{\H{o}}s--Ko--Rado theorems: algebraic approaches}.
\newblock Cambridge University Press, 2016.

\bibitem{hilton1967some}
A.~J. Hilton and E.~C. Milner.
\newblock Some intersection theorems for systems of finite sets.
\newblock {\em The Quarterly Journal of Mathematics}, 18(1):369--384, 1967.

\bibitem{keller2024t}
N.~Keller, N.~Lifshitz, D.~Minzer, and O.~Sheinfeld.
\newblock On t-intersecting families of permutations.
\newblock {\em Advances in Mathematics}, 445:109650, 2024.

\bibitem{kostochka2017structure}
A.~Kostochka and D.~Mubayi.
\newblock The structure of large intersecting families.
\newblock {\em Proceedings of the American Mathematical Society},
  145(6):2311--2321, 2017.

\bibitem{li2020erd}
C.~H. Li, S.~J. Song, and V.~R.~T. Pantangi.
\newblock {E}rd{\H{o}}s--{K}o--{R}ado problems for permutation groups.
\newblock {\em arXiv preprint arXiv:2006.10339}, 2020.

\bibitem{liao2023nontrivial}
J.~Liao, Z.~Lv, M.~Cao, and M.~Lu.
\newblock Nontrivial intersecting families in multisets.
\newblock {\em arXiv preprint arXiv:2308.03585}, 2023.

\bibitem{liao2024erdHos}
J.~Liao, Z.~Lv, M.~Cao, and M.~Lu.
\newblock {E}rd{\H{o}}s-{K}o-{R}ado theorem for bounded multisets.
\newblock {\em Journal of Combinatorial Theory, Series A}, 206:105888, 2024.

\bibitem{meagher2011erdHos}
K.~Meagher and A.~Purdy.
\newblock An {E}rd{\H{o}}s-{K}o-{R}ado theorem for multisets.
\newblock {\em The Electronic Journal of Combinatorics}, pages P220--P220,
  2011.

\bibitem{meagher2016intersection}
K.~Meagher and A.~Purdy.
\newblock Intersection theorems for multisets.
\newblock {\em European Journal of Combinatorics}, 52:120--135, 2016.

\bibitem{meagher2016erdHos}
K.~Meagher, P.~Spiga, and P.~H. Tiep.
\newblock An {E}rd{\H{o}}s--{K}o--{R}ado theorem for finite 2-transitive
  groups.
\newblock {\em European Journal of Combinatorics}, 55:100--118, 2016.

\bibitem{rosen2013number}
M.~Rosen.
\newblock {\em Number theory in function fields}, volume 210.
\newblock Springer Science \& Business Media, 2013.

\bibitem{CT}
C.~Tompkins.
\newblock Personal communication, 2018.

\end{thebibliography}
\end{document}